\documentclass[reqno]{amsart} 
 
\usepackage{hyperref}
\usepackage{amsmath}
\usepackage{amssymb}
\usepackage{amsfonts}
\usepackage{amsthm}
\usepackage{latexsym}
\usepackage{esint}
\usepackage{mathtools}
\usepackage{mathrsfs}
\usepackage{graphicx}
\usepackage{wrapfig}
\usepackage{bbm}
\usepackage{color}
\usepackage{bm}
\usepackage[top=1in, bottom=1.25in, left=1.10in, right=1.10in]{geometry}
\usepackage{todonotes}
\usepackage{stackengine}
\allowdisplaybreaks

\newtheorem{theorem}{Theorem}[section]
\newtheorem{lemma}[theorem]{Lemma}

\theoremstyle{definition}

\theoremstyle{remark}
\newtheorem{remark}[theorem]{Remark}

\numberwithin{equation}{section}

\newcommand{\bx}{\mathbf{x}}

\newcommand{\bz}{\mathbf{z}}

\newcommand{\bk}{\mathbf{k}}
\newcommand{\bj}{\mathbf{j}}
\newcommand{\bl}{\bm{\ell}}
\newcommand{\bff}{\mathbf{f}}
\newcommand{\bg}{\mathbf{g}}
\newcommand{\bu}{\mathbf{u}}

\newcommand{\bv}{\mathbf{v}}
\newcommand{\bw}{\mathbf{w}}
\newcommand{\bn}{\mathbf{n}}

\newcommand{\dd}{\,\mathrm{d}}

\newcommand{\dx}{\, \mathrm{d} \mathbf{x}}

\newcommand{\ds}{\, \mathrm{d}s}

\newcommand{\divx}{\mathrm{div}_{\mathbf{x}}}

\begin{document}


\title[ ]{Complex analytic solutions for the {TQG} model}

%
%
\author{Prince Romeo Mensah}
\address{Institut f\"ur Mathematik,
Technische Universit\"at Clausthal,
Erzstra{\ss}e 1,
38678 Clausthal-Zellerfeld, Germany}
\email{prince.romeo.mensah@tu-clausthal.de \\ orcid ID: 0000-0003-4086-2708}

\subjclass[2020]{76B03, 35B65, 37N10}

\date{\today}


\keywords{Thermal quasi-geostrophic, Oceanography,  Holomorphic function, Gevrey class, Analyticity}

\begin{abstract} 
We present a condition under which the thermal quasi-geostrophic (TQG) model possesses a solution that is holomorphic in time with values in the Gevrey space of complex analytic functions. This can be seen as the complex extension of the work by Levermore and Oliver (1997) for the generalised Euler equation but applied to the TQG model.

\end{abstract}

\maketitle
\section{Introduction}
\label{sec:intro}
There are several 2-dimensional models for geophysical fluid flows where horizontal length scales are very large compared to their vertical scale. One such model is the Thermal Quasi-Geostrophic (TQG) model where there is a near balance between the Coriolis force, the hydrostatic pressure gradient force and the buoyancy gradient force. We refer to \cite{beron2021geometry, beron2021nonlinear, holm2021stochastic} for the historical account, the derivation and some simulations. In precise form, the TQG model posed on the space-time spatially periodic cylinder  $I\times\mathbb{T}^2$ where $I:=(0,T), T>0$,
is a coupled system of equations governed by the evolution of the buoyancy $b:(t,\bx)\in I\times\mathbb{T}^2\mapsto b(t,\bx)\in \mathbb{R}$ and the potential vorticity $q:(t,\bx)\in I\times\mathbb{T}^2\mapsto q(t,\bx)\in \mathbb{R}$ in the following way:
\begin{align}
\partial_t b + (\bu \cdot \nabla) b =0,
\label{ce}
\\
\partial_t q + (\bu\cdot \nabla)( q -b) = -(\bu_h \cdot \nabla) b,  \label{me}
\\
b_0(x)=b(0,x),\qquad q_0(x)=q(0,x),
\label{initialTQG}
\end{align}
where
\begin{align}
\label{constrt}
\bu =\nabla^\perp \psi, 
\qquad
\bu_h =\frac{1}{2} \nabla^\perp h,
\qquad
q=(\Delta -1) \psi +f.
\end{align}
In \eqref{constrt}, the vector $\bu:(t,\bx)\in I\times\mathbb{T}^2\mapsto \bu(t,\bx)\in \mathbb{R}^2$ is the velocity field, $\psi:(t,\bx)\in I\times\mathbb{T}^2\mapsto \psi(t,\bx)\in \mathbb{R}$ is the streamfunction, $h:(t,\bx)\in I\times\mathbb{T}^2\mapsto h(t,\bx)\in \mathbb{R}$ is the spatial variation around a constant bathymetry
profile and $f:(t,\bx)\in I\times\mathbb{T}^2\mapsto f(t,\bx)\in \mathbb{R}$ is the spatial variation around a constant background rotation rate. Note that since $\nabla^\perp=(-\partial_{x_2},\partial_{x_1})$, both $\bu$ and $\bu_h$ are incompressible vector fields.

The existence of unique local-in-time strong solutions of \eqref{ce}-\eqref{constrt} have recently been shown in \cite{crisan2023theoretical1} and in \cite{crisan2023theoretical2} for the stochastic variant. Here, `strong solutions' are those having a minimal spatial regularity of $(b(t), q(t))\in W^{3,2}(\mathbb{T}^2)\times W^{2,2}(\mathbb{T}^2)$ for a corresponding dataset $(\bu_h,f,b_0,q_0)$ satisfying
\begin{align*}
(\bu_h, f)\in W^{3,2}_{\divx}(\mathbb{T}^2)\times W^{2,2}(\mathbb{T}^2),
\qquad
(b_0, q_0)\in W^{3,2}(\mathbb{T}^2)\times W^{2,2}(\mathbb{T}^2) 
\end{align*}
where $W^{3,2}_{\divx}(\mathbb{T}^2)$ is the space of divergence-free vector fields in $W^{3,2}(\mathbb{T}^2)$.
Since we work on a torus, one can use the standard method of taking partial spatial derivative of the equation and applying commutator estimates to the advection terms (see \cite[Proposition 2.1]{majda1984compressible} and \cite[Appendix]{klainerman1981singular} for instance) to extend this solution to
\begin{align*}
(b(t), q(t))\in W^{k+1,2}(\mathbb{T}^2)\times W^{k,2}(\mathbb{T}^2)\quad \text{for all } k\geq2.
\end{align*}
Consequently, smooth solutions $(b(t), q(t))\in C^{\infty}(\mathbb{T}^2)\times C^{\infty}(\mathbb{T}^2)$ also exist locally in time provided the data are also smooth. The existence of solutions in the larger class $(b(t), q(t))\in W^{k+1,2}(\mathbb{T}^2)\times W^{k,2}(\mathbb{T}^2)$ for any $k<2$ as well as the existence of global-in-time solution in any class, however, remain interesting open problems.
 
In this work, we are interested in studying the analytical property of smooth solutions of \eqref{ce}-\eqref{constrt}. In particular, we give a condition under which smooth solutions have convergent Taylor series. This will follow by showing that smooth solutions belong to the \textit{Gevrey class} of analytic functions under the condition that the radius of analyticity solves a suitable ordinary differential equation. The precise statement of this result is given below in Theorem \ref{thm:main}. See \cite{biswas2014gevrey} for a related result for the supercritical surface quasi-geostrophic equation,
\cite{kukavica2010existence} for the hydrostatic Euler equations and \cite{bardos1976analyticity, biswas2022space, cappiello2015some, kukavica2009radius, kukavica2011domain, levermore1997analyticity} for the (generalised) Euler equation

\section{Preliminaries}
\label{sec:prelim}
\subsection{Notation}
We consider a space-time cylinder consisting of  spatial points $ \mathbf{x}=(x_1,\ldots, x_d) $ on the  torus $\mathbb{T}^d=(\mathbb{R}/2\pi\mathbb{Z})^d$, $d=2,3$ with periodic boundary condition  and a time variable  $t\in [0,T]$ where $T>0$ is fixed and arbitrary. For functions $F$ and $G$, we write $F \lesssim G$  if there exists  a generic constant $c>0$  such that $F \leq c\,G$.
We also write $F \lesssim_p G$ if the  constant  $c(p)>0$ depends on a variable $p$. If $F \lesssim G$ and $G\lesssim F$ both hold (respectively,  $F \lesssim_p G$ and $G\lesssim_p F$), we use the notation $F\sim G$ (respectively, $F\sim_p G$).
The symbol $\vert \cdot \vert$ may be used in four different contexts. For a scalar function $f\in \mathbb{R}$, $\vert f\vert$ denotes the absolute value of $f$. For a vector $\bff\in \mathbb{R}^d$, $\vert \bff \vert$ denotes the Euclidean norm of $\bff$. For a square matrix $\mathbb{F}\in \mathbb{R}^{d\times d}$, $\vert \mathbb{F} \vert$ shall denote the Frobenius norm $\sqrt{\mathrm{trace}(\mathbb{F}^T\mathbb{F})}$. Also, if $S\subseteq  \mathbb{R}^d$ is  a (sub)set, then $\vert S \vert$ is the $d$-dimensional Lebesgue measure of $S$. Lastly, for any two normed spaces $(V_1,\Vert \cdot\Vert_{V_1})$ and  $(V_2,\Vert \cdot\Vert_{V_2})$ with $v_1\in V_1$ and $v_2\in V_2$, we let
\begin{align*}
\Vert (v_1\,,\,v_2)\Vert_{V_1\times V_2}^p:=\Vert v_1\Vert_{V_1}^p+\Vert v_2\Vert_{V_2}^p
\end{align*}
represent the product norm for any $p\in[1,\infty)$.

\subsection{The functional setting}
Let $C^\infty_{\divx}(\mathbb{T}^d)$ denote the space of all divergence-free smooth periodic functions. For $r\geq0$, let the Hilbert space $(H^r(\mathbb{T}^d), \Vert \cdot\Vert_{H^r})$ represent the $L^2$-homogeneous Sobolev space of mean-free functions  defined as
\begin{equation}
\begin{aligned}
\label{hsSpacehomoXX}
 H^r(\mathbb{T}^d)=\Bigg\{
\bff(\bx)=& \sum_{\bk \in \mathbb{Z}^d_{\bm{0}}}\hat{\bff}_{\bk} e^{i\bk\cdot\bx}
 \in
L^2(\mathbb{T}^d)
\, \bigg\vert\,
\hat{\bff}_{\bk}
=\int_{\mathbb{T}^d} e^{-i\bk \cdot\bx}\bff(\bx) \dx \in \mathbb{C}^d,
\\&
  \overline{\hat{\bff}}_{\bk}= \hat{\bff}_{-\bk}, \quad \hat{\bff}_{\bm{0}}=\bm{0}, \quad 
\Vert \bff \Vert_{ H^r} =\bigg( \sum_{\bk\in\mathbb{Z}^d_{\bm{0}}}  \vert \bk\vert^{2r} \vert\hat{\bff}_{\bk}\vert^2 \bigg)^{1/2}<\infty
\Bigg\}.
\end{aligned}
\end{equation}
Here, the $\hat{\bff}_{\bk}$s are   the Fourier coefficients of the function $\bff\in L^2(\mathbb{T}^d)$ and $\mathbb{Z}^d_{\bm{0}}:= 2 \pi \mathbb{Z}^d\setminus\{\bm{0}\}$. Furthermore, note that $\hat{\bff}_{\bm{0}}=\bm{0}$ is equivalent to saying that elements $\bff\in  H^r(\mathbb{T}^d)$ in  \eqref{hsSpacehomoXX} are mean-free, i.e., $\int_{\mathbb{T}^d}\bff(\bx)\dx=0$.

\noindent With \eqref{hsSpacehomoXX} in hand, we now define  $(\mathbb{H}^r(\mathbb{T}^d), \Vert \cdot\Vert_{H^r})$ as the subclass of $(H^r(\mathbb{T}^d), \Vert \cdot\Vert_{H^r})$ given by
\begin{equation}
\begin{aligned}
\label{hsSpacehomoYY}
 \mathbb{H}^r(\mathbb{T}^d)=\bigg\{
\bff(\bx)&= \sum_{\bk \in \mathbb{Z}^d_{\bm{0}}}\hat{\bff}_{\bk} e^{i\bk\cdot\bx}
\in
H^r(\mathbb{T}^d)
\, \big\vert\,
 \hat{\bff}_\bk\cdot \bk=0
\bigg\}.
\end{aligned}
\end{equation}
Here, the condition $\hat{\bff}_\bk\cdot \bk=0$ corresponds to the divergence-free condition $\divx \bff=0$ in frequency space. Thus,  $\mathbb{H}^r(\mathbb{T}^d) $ is the subclass of divergence-free vectors in $H^r(\mathbb{T}^d)$.
Next, if we consider the positive, self-adjoint  operator $\Lambda=-\Delta$ which is known to possess a nondecreasing sequence  $\{\lambda_j\}_{j\in \mathbb{N}}$ of strictly positive eigenvalues that approach infinity as $j\rightarrow\infty$, see \cite{lions1996mathematical}. Then for any $r\in \mathbb{R}$, we can define its fractional power $\Lambda^r$ as the mapping
\begin{align*}
\bff(\bx)&= \sum_{\bk \in \mathbb{Z}^d_{\bm{0}}}\hat{\bff}_{\bk} e^{i\bk\cdot\bx}
\quad \mapsto \quad 
\Lambda^r\bff
=
\sum_{\bk\in \mathbb{Z}^d_{\bm{0}}}  \vert \bk\vert^{2r}  \hat{\bff}_{\bk} e^{i\bk\cdot\bx}
= \sum_{\bk\in \mathbb{Z}^d_{\bm{0}}}  \lambda_{\vert\bk\vert}^r  \langle\bff(\bx), e^{i\bk \cdot\bx}\rangle_{L^2} e^{i\bk\cdot\bx}
\end{align*}
with $\Lambda^0=I$ so that
\begin{align*} 
\Vert \Lambda^r\bff\Vert_{L^2}^2= \sum_{\bk\in \mathbb{Z}^d_{\bm{0}}}  \vert \bk\vert^{4r} \vert \hat{\bff}_{\bk} \vert^2
= \sum_{\bk\in \mathbb{Z}^d_{\bm{0}}}  \lambda_{\vert\bk\vert}^{2r}\vert  \hat{\bff}_{\bk} \vert^2
=
\Vert \bff \Vert_{ H^{2r}}^2.
\end{align*}

\noindent A special subclass of the space  \eqref{hsSpacehomoYY}  are the Gevrey-class-$(s,r,\varphi)$ spaces $G^{s,r}_{\varphi }(\mathbb{T}^d)$ given by
\begin{align*}
D(e^{\varphi  \Lambda^{1/2s}} : \mathbb{H}^r(\mathbb{T}^d)) \equiv \big\{ \bff(\bx)\in \mathbb{H}^r(\mathbb{T}^d) \, : \, \Vert e^{\varphi  \Lambda^{1/2s}} \bff(\bx)  \Vert_{H^r}<\infty \big \}
\end{align*}
for real numbers $s>0$, $r\geq 0$ and a given time-dependent function $\varphi:I \rightarrow [0,\infty)$. The space \eqref{hsSpacehomoYY} is recovered when $\varphi=0$ whereas 
\begin{align*}
D(e^{\varphi  \Lambda^{1/2s}} : \mathbb{H}^r(\mathbb{T}^d))
\subseteq
C^\infty_{\divx}(\mathbb{T}^d) \subseteq \mathbb{H}^r(\mathbb{T}^d)
\end{align*}
when $\varphi>0$. A similar relation holds between $D(e^{\varphi  \Lambda^{1/2s}} : H^r(\mathbb{T}^d))$, $C^\infty(\mathbb{T}^d)$ and $H^r(\mathbb{T}^d)$ as defined in \eqref{hsSpacehomoXX}. See \cite{levermore1997analyticity} for further details. In particular, the spaces $D(e^{\varphi  \Lambda^{1/2s}} : H^r(\mathbb{T}^d))$ are nested just as  $H^r(\mathbb{T}^d)$ are, in the sense that the embedding $D(e^{\varphi  \Lambda^{1/2s}} : H^{r_2}(\mathbb{T}^d)) \hookrightarrow D(e^{\varphi  \Lambda^{1/2s}} : H^{r_1}(\mathbb{T}^d))$ is always continuous for $r_1\leq r_2$.

With the above preliminaries in hand, we now present some useful results for this work. Firstly,
letting  $\langle\cdot,\cdot\rangle$ represent the $L^2$ inner product  we have the following generalisation of the crucial result from \cite[Lemma 8]{levermore1997analyticity}.
\begin{lemma}
\label{lem:ConvEst}
For $\bu,\bv\in D( e^{\varphi \Lambda^{1/2}} :  {H}^{r+1/2}(\mathbb{T}^d))$ and $\bw\in D( e^{\varphi \Lambda^{1/2}} :  {H}^{r+1}(\mathbb{T}^d))$  with $r>d/2+3/2$, we have that
\begin{align*}
\vert\langle  \Lambda^{r/2}e^{\varphi \Lambda^{1/2}}& (\bu\cdot\nabla \bv) \,,\, \Lambda^{r/2}e^{\varphi \Lambda^{1/2}} \bw \rangle\vert
\lesssim 
\Vert \bu\Vert_{H^r}
 \Vert\bv\Vert_{H^r}  
\Vert \bw\Vert_{H^r} 
 \\&+
\varphi 
\Big(
\Vert e^{\varphi \Lambda^{1/2}} \bv\Vert_{H^r}
 \Vert e^{\varphi \Lambda^{1/2}}\bu\Vert_{H^{r+\frac{1}{2}}}  
+
\Vert e^{\varphi \Lambda^{1/2}} \bu\Vert_{H^{r }}
 \Vert e^{\varphi \Lambda^{1/2}}\bv\Vert_{H^{r+\frac{1}{2}}} 
 \Big) 
\Vert  e^{\varphi \Lambda^{1/2}}\bw\Vert_{H^{r+\frac{1}{2}}} 
\end{align*}
\end{lemma} 

\begin{remark}
Lemma \ref{lem:ConvEst} generalises \cite[Lemma 8]{levermore1997analyticity} in two ways and we defer its proof to the appendix below. Firstly, nowhere is the incompressibility condition used (hence  $\bu\in D( e^{\varphi \Lambda^{1/2}} :  {H}^{r+1/2}(\mathbb{T}^d))$ rather than  $\bu\in D( e^{\varphi \Lambda^{1/2}} :  \mathbb{H}^{r+1/2}(\mathbb{T}^d))$) whereas $\bu$ was taken to be (weighted) incompressible in the cited result. 
Secondly, we clearly do not require $\bv=\bw$ nor do they need to be scaler-valued.
\end{remark} 
Finally, we recall the following result in \cite[Lemma 2.2]{crisan2023theoretical1}. Although the original result is stated for non-negative integers $k\in\mathbb{N}\cup\{0\}$, a look at the proof shows that it extends to all non-negative real numbers $r\geq0$ with no further modification to the proof.
\begin{lemma}
\label{lem:VelToVor}
Let $r\geq0$ and assume that the triple $(\psi,\bu,\omega)$ satisfies
\begin{align*}
\bu=\nabla^\perp\psi, \qquad \omega=(\Delta-1)\psi.
\end{align*}
If $\omega\in H^r(\mathbb{T}^2)$, then the following estimate
\begin{align*}
\Vert \bu \Vert_{H^{r+1}}^2\lesssim \Vert\omega\Vert_{H^r}^2
\end{align*}
holds.
\end{lemma} 
\section{Main result}
\subsection{Complexified TQG}  
We wish to find a complex neighbourhood of the initial time where we can find a unique solution of the complexified TQG that is holomorphic in time with values in the Gevrey space of complex analytic functions.
For us to extend the TQG equation to a complex time with complex-valued solutions $(b,q)=(b_1+ib_2,q_1+iq_2)$ where $(b_1,q_1)$ and $(b_2,q_2)$ are real-valued solutions, we consider the complexification of the all the functions spaces under consideration \cite{luna2012complexifications}. For example, the complexification of the domain of $\Lambda$, $D(\Lambda)$, is $D(\Lambda)_{\mathbb{C}}:=D(\Lambda)+i D(\Lambda)$ and so on. Furthermore, for $\bff,\bg\in \mathbb{C}$ given by $\bff=\bff_1+i\bff_2$ and $\bg=\bg_1+i\bg_2$, the complexified $L^2$-inner product is given by
\begin{align*}
\langle\bff\,,\, \bg\rangle_\mathbb{C}
=
\langle\bff_1\,,\, \bg_1\rangle
+
\langle\bff_2\,,\, \bg_2\rangle
+i
[\langle\bff_2\,,\, \bg_1\rangle
-
\langle\bff_1\,,\, \bg_2\rangle]
\end{align*}
whereas the complexified bilinear form $B(\bff\,,\, \bg):=(\bff \cdot\nabla)\bg$
is given by the  bilinear form
\begin{align*}
B(\bff\,,\, \bg)_\mathbb{C}
= 
B(\bff_1\,,\, \bg_1)
-
B(\bff_2\,,\, \bg_2)
+i
[B(\bff_2\,,\, \bg_1)
+
B(\bff_1\,,\, \bg_2)].
\end{align*}
Correspondingly, the complexified trilinear form is given by
\begin{align*}
\langle B(\bff\,,\, \bg)_\mathbb{C}\,,\, \mathbf{h} \rangle_\mathbb{C}
=&
\langle B(\bff_1\,,\, \bg_1)-B(\bff_2\,,\, \bg_2)\,,\, \mathbf{h} _1\rangle
+
\langle B(\bff_2\,,\, \bg_1)+B(\bff_1\,,\, \bg_2) \,,\, \mathbf{h} _2\rangle
\\&+
i\big[
\langle B(\bff_2\,,\, \bg_1)+B(\bff_1\,,\, \bg_2) \,,\, \mathbf{h} _1\rangle
-
\langle B(\bff_1\,,\, \bg_1)-B(\bff_2\,,\, \bg_2) \,,\, \mathbf{h} _2\rangle
\big].
\end{align*}
Note that since we work on a periodic domain, if $\bff$ is divergence free (i.e. $\bff_1$ and $\bff_2$ are both divergence free) and $\bg=\mathbf{h}$ (i.e. $\bg_1=\mathbf{h}_1$, and $\bg_2=\mathbf{h}_2$), then the orthogonality property of the real-valued trilinear form means that 
\begin{align*}
\langle B(\bff\,,\, \bg)_\mathbb{C}\,,\, \bg \rangle_\mathbb{C}
=&-
\langle B(\bff_2\,,\, \bg_2)\,,\, \bg_1\rangle
+
\langle B(\bff_2\,,\, \bg_1) \,,\, \bg _2\rangle
\\&+
i\big[
\langle B(\bff_1\,,\, \bg_2) \,,\, \bg _1\rangle
-
\langle B(\bff_1\,,\, \bg_1)  \,,\, \bg_2\rangle
\big].
\end{align*}
Thus, the complexified trilinear form is not orthogonal in its last two variables. With this preparation in hand, we fix $\theta\in(\frac{\pi}{2},\frac{\pi}{2})$ and consider the time $\zeta=se^{i\theta}$ for $s\in I$ so that $\mathrm{Re} \,\zeta =s\cos(\theta)\in I$. Here, without loss of generality, we assume that $I=(0,T)$ is such that $T>0$ is the maximal time for the existence of smooth solutions.
The complexified TQG equation is given by
\begin{align}
\frac{\dd b}{\dd\zeta}  + B(\bu,b)_\mathbb{C} =0,
\label{ceCx}
\\
\frac{\dd q}{\dd\zeta}  + B(\bu,q)_\mathbb{C} -B(\bu,b)_\mathbb{C}  = -B(\bu_h,b)_\mathbb{C} ,  \label{meCx}
\\
b_0(x)=b(0,x)_\mathbb{C},\qquad q_0(x)=q(0,x)_\mathbb{C},
\label{initialTQGCx}
\end{align}
with
\begin{align}
\label{constrtCx}
\bu =\nabla^\perp \psi_\mathbb{C}, 
\qquad
\bu_h =\frac{1}{2} \nabla^\perp h_\mathbb{C},
\qquad
q=(\Delta -1) \psi_\mathbb{C} +f_\mathbb{C}
\end{align}
where $g_\mathbb{C}:=g_1+ig_2$ for $g_1,g_2\in \mathbb{R}$. Our main result is now given by the following.

\begin{theorem}
\label{thm:main}
Let $\varphi_0>0$ be a constant and let $(\bu_h,f,b_0,q_0)$ be a dataset satisfying 
\begin{align}
&(b_0,q_0) \in D( e^{\varphi_0 \Lambda^{1/2}} : {H}^4(\mathbb{T}^2))_\mathbb{C}\times  D( e^{\varphi_0 \Lambda^{1/2}}: {H}^3(\mathbb{T}^2))_\mathbb{C},\label{initialReg}
\\
&(\bu_h,f) \in  D( e^{\varphi_0 \Lambda^{1/2}}: \mathbb{H}^{7/2}(\mathbb{T}^2))_\mathbb{C} \times  D( e^{\varphi_0 \Lambda^{1/2}}: {H}^{7/2}(\mathbb{T}^2))_\mathbb{C}.
\label{forcingReg}
\end{align}
 Fix $\theta\in(-\frac{\pi}{2},\frac{\pi}{2})$ and for $s\in I$, let  $\varphi =\varphi(s\cos\theta)$ be such that $\varphi\in C^1(\overline{I})$. 
If the function $\varphi $ solves
\begin{align}
\label{varphiFunk}
\varphi(s)=\varphi_0e^{-c\int_0^s \Theta(\tau)\dd\tau}, \qquad \varphi(s=0)=\varphi_0
\end{align}
for $c>0$, where
\begin{align*}
\Theta(s):=& 
\Vert( e^{\varphi(s\cos\theta) \Lambda^{1/2}} b(\zeta,\cdot)\,,\,e^{\varphi(s\cos\theta) \Lambda^{1/2}}q(\zeta,\cdot))\Vert_{H^{4}_{\mathbb{C}} \times H^{3}_{\mathbb{C}}}
  \\&+
 \Vert (e^{\varphi_0  \Lambda^{1/2}  } \bu_h(\cdot) \,,\, e^{\varphi_0 \Lambda^{1/2}}f(\cdot))\Vert_{H_{\mathbb{C}}^{3} \times H^{3}_{\mathbb{C}}},
\end{align*}
$\zeta=se^{i\theta}$,
then a dataset $(\bu_h,f,b_0,q_0)$   satisfying \eqref{initialReg}-\eqref{forcingReg} generates  a unique solution of \eqref{ceCx}-\eqref{constrtCx} in the region 
\begin{align*}
\left\{\zeta=se^{i\theta}\quad :\quad \mathrm{Re} \,\zeta \in I, \quad
\vert\theta\vert<\frac{\pi}{2}, \quad 0<s\, \mathcal{D}(\mathrm{data})^\frac{1}{2}< \frac{1}{c\cos\theta} \right\}
\end{align*}
where
\begin{align*}
\mathcal{D}(\mathrm{data}):=\Vert    (e^{\varphi_0  \Lambda^{1/2}  } b_0 \,,\,e^{\varphi_0   \Lambda^{1/2}  } q_0  ) \Vert^2_{H_{\mathbb{C}}^4\times H_{\mathbb{C}}^3} 
+
\Vert    (e^{\varphi_0  \Lambda^{1/2}  } \bu_h \,,\,e^{\varphi_0   \Lambda^{1/2}  } f  ) \Vert_{H_{\mathbb{C}}^{7/2} \times H^{7/2}_{\mathbb{C}}}^2.
\end{align*}
\end{theorem}
\begin{remark}
Theorem \ref{thm:main} remains true if the family of  spaces $\{H^3, H^{7/2}, \mathbb{H}^{7/2}, H^4\}$ (and their complexified variants) are replaced with  $\{H^r, H^{r+\frac{1}{2}}, \mathbb{H}^{r+\frac{1}{2}}, H^{r+1}\}$ (and their complexified variants), respectively, provided that $r>\frac{5}{2}$. This is consequence of Lemma \ref{lem:ConvEst}. The choice of the former family is to make the analysis clear.
\end{remark}
\begin{remark}
Under the Beale--Kato--Majda criterion for the blow-up of solution to the TQG \cite{crisan2021blow},  the time span $T>0$ can be made arbitrarily large. In this case, the complex neighbourhood where the solution is of Gevrey analytic class becomes
\begin{align*}
\left\{\zeta=se^{i\theta}\quad : \quad
\mathrm{Re} \,\zeta>0, \quad
\vert\theta\vert<\frac{\pi}{2}, \quad 0<s\, \mathcal{D}(\mathrm{data})^{1/2}< \frac{1}{c\cos\theta} \right\}
\end{align*}
so that by tuning $\theta$, one can make $s$ arbitrarily large.
\end{remark}
\subsection{Proof of result}
We now present the proof of Theorem \ref{thm:main} which uses arguments in  \cite{crisan2022spatial, foias1989gevrey, levermore1997analyticity} and \cite[Chapter II]{foias2001navier}. This will involve formal a priori estimates which, as usual, can be made rigorous by first considering a finite-dimensional Galerkin approximation and later passing to the limit in the approximation parameter.
With this said, for $\varphi:=\varphi(\mathrm{Re} \,\zeta)=\varphi(s\cos\theta)$ with fixed $\theta\in(-\frac{\pi}{2},\frac{\pi}{2})$, we first consider the following equation that is due to the product rule for the product of two functions:

\begin{align*}
\frac{1}{2}\frac{\dd}{\ds} \Vert    e^{\varphi  \Lambda^{1/2}  } b \Vert^2_{H_{\mathbb{C}}^4} 
=& 
\mathrm{Re}\, 
\Big\langle  \Lambda^{2 }
\frac{\dd}{\ds}\big(
 e^{\varphi(s\cos\theta)  \Lambda^{1/2}  } b(\zeta)\big)
,\,     \Lambda^{2 } e^{\varphi(s\cos\theta)    \Lambda^{1/2}  } b(\zeta)
\Big\rangle _{\mathbb{C}}
\\
=&
\cos\theta\,\Dot{\varphi} 
 \Vert    e^{\varphi  \Lambda^{1/2}  } b(\zeta) \Vert^2_{H_{\mathbb{C}}^{9/2}} 
+
 \mathrm{Re}\,
 \Big\langle 
 \Lambda^{2} e^{\varphi  \Lambda^{1/2}  }e^{i\theta} \tfrac{\dd b(\zeta)}{\dd\zeta}
,\,     \Lambda^{2} e^{\varphi  \Lambda^{1/2}  }b(\zeta)
\Big\rangle_{\mathbb{C}} 
\\
=&
\cos\theta\,
\Dot{\varphi} 
 \Vert    e^{\varphi  \Lambda^{1/2}  } b \Vert^2_{H_{\mathbb{C}}^{9/2}} 
-
 \mathrm{Re}\,
 e^{i\theta}
 \Big\langle
  \Lambda^{2} e^{\varphi  \Lambda^{1/2}  }
B(\bu,b)_{\mathbb{C}}
,\,     \Lambda^{2} e^{\varphi  \Lambda^{1/2}  }b \Big\rangle_{\mathbb{C}}.
\end{align*}
Similarly, we obtain
\begin{align*}
\frac{1}{2}\frac{\dd}{\ds} \Vert    e^{\varphi  \Lambda^{1/2}  } q \Vert^2_{H_{\mathbb{C}}^3} 
=&
\cos\theta\,\Dot{\varphi} 
 \Vert    e^{\varphi  \Lambda^{1/2}  } q \Vert^2_{H_{\mathbb{C}}^{7/2}} 
-
 \mathrm{Re}\,
 e^{i\theta}
 \Big\langle
 \Lambda^{3/2 } e^{\varphi  \Lambda^{1/2}  }B(\bu,q)_{\mathbb{C}}
,\,     \Lambda^{3/2 } e^{\varphi  \Lambda^{1/2}  }q 
\Big\rangle_{\mathbb{C}}
\\&+
 \mathrm{Re}\,
 e^{i\theta}
 \Big\langle
 \Lambda^{3/2 } e^{\varphi  \Lambda^{1/2}  } B((\bu-\bu_h),b)_{\mathbb{C}}
,\,     \Lambda^{3/2 } e^{\varphi  \Lambda^{1/2}  }q 
\Big\rangle_{\mathbb{C}}.
\end{align*}
Thus, by combining the two equations above, we obtain
\begin{align*}
\frac{1}{2}\frac{\dd}{\ds} \Vert    (e^{\varphi  \Lambda^{1/2}  } b \,,\,e^{\varphi  \Lambda^{1/2}  } q ) \Vert^2_{H_{\mathbb{C}}^4\times H_{\mathbb{C}}^3} 
\leq&
\cos\theta\,\Dot{\varphi} 
\Vert    (e^{\varphi  \Lambda^{1/2}  } b(\zeta) \,,\,e^{\varphi  \Lambda^{1/2}  } q(\zeta)) \Vert^2_{H_{\mathbb{C}}^{9/2}\times H_{\mathbb{C}}^{7/2}} 
\\&
+
\cos \theta\big(\big\vert 
 \big\langle
 \Lambda^{2} e^{ \varphi  \Lambda^{1/2}  }B(\bu,b)_{\mathbb{C}}
,\,     \Lambda^{2} e^{ \varphi  \Lambda^{1/2}  }b \big\rangle_{\mathbb{C}}
\big\vert
\\&
\qquad+\big\vert 
 \big\langle
 \Lambda^\frac{3}{2 } e^{\varphi  \Lambda^{1/2}  }B(\bu,q)_{\mathbb{C}}
,\,     \Lambda^\frac{3}{2 } e^{\varphi  \Lambda^{1/2}  }q 
\big\rangle_{\mathbb{C}}
\big\vert
\\&\qquad+
\big\vert 
 \big\langle
 \Lambda^\frac{3}{2 } e^{\varphi  \Lambda^{1/2}  } B((\bu-\bu_h),b)_{\mathbb{C}}
,\,     \Lambda^\frac{3}{2 } e^{\varphi  \Lambda^{1/2}  }q 
\big\rangle_{\mathbb{C}}\big\vert\big).
\end{align*}
Now, by using Lemma \ref{lem:ConvEst} and the fact that for any complex $f=f_1+if_2$ in a complexified normed space $(X_{\mathbb{C}}, \Vert\cdot\Vert_{X_{\mathbb{C}}})$ the estimate $\Vert f_i\Vert_X\leq \Vert f\Vert_{X_\mathbb{C}}$ hold for any $i=\{1,2\}$, we obtain
\begin{align*}
 \big\vert 
 \big\langle
 \Lambda^{2} e^{ \varphi  \Lambda^{1/2}  }B(\bu,b)_{\mathbb{C}}
,\,     \Lambda^{2} e^{ \varphi  \Lambda^{1/2}  }b \big\rangle_{\mathbb{C}}
\big\vert
\lesssim&
\Vert \bu\Vert_{H^4_{\mathbb{C}}}
 \Vert b\Vert_{H^4_{\mathbb{C}}}^2 
+
\varphi
\big(
 \Vert e^{\varphi \Lambda^{1/2}} \bu\Vert_{H^{4}_{\mathbb{C}}}
 \Vert e^{\varphi \Lambda^{1/2}}b\Vert_{H^{9/2}_{\mathbb{C}}}^2 
\\&
\quad
+
 \Vert e^{\varphi \Lambda^{1/2}} \bu\Vert_{H^{9/2}_{\mathbb{C}}}
 \Vert e^{\varphi \Lambda^{1/2}}b\Vert_{H^{4}_{\mathbb{C}}}
  \Vert e^{\varphi \Lambda^{1/2}}b\Vert_{H^{9/2}_{\mathbb{C}}}\big)
\\
\lesssim&
\Vert b\Vert_{H^4_{\mathbb{C}}}^3
+
 \Vert q\Vert_{H^3_{\mathbb{C}}}^3 
+
 \Vert f\Vert_{H^3_{\mathbb{C}}}^3 
\\&+
\varphi
\big(
\Vert ( e^{\varphi \Lambda^{1/2}} b\,,\,  e^{\varphi \Lambda^{1/2}}q)\Vert_{H^{4}_{\mathbb{C}} \times H^{3}_{\mathbb{C}}}
  +
\Vert  e^{\varphi_0 \Lambda^{1/2}}f\Vert_{H^{3}_{\mathbb{C}}}
\big)
\\&\times
\big(
 \Vert ( e^{\varphi \Lambda^{1/2}} b\,,\, e^{\varphi \Lambda^{1/2}}q)\Vert_{  H^{9/2}_{\mathbb{C}} \times H^{7/2}_{\mathbb{C}}}^2
 +
 \Vert  e^{\varphi_0 \Lambda^{1/2}}f \Vert_{    H^{7/2}_{\mathbb{C}}}^2
 \big)
\end{align*}
where in the second inequality, we apply Lemma \ref{lem:VelToVor} to \eqref{constrtCx} and use Young's equality. We  also use the assumption $\varphi(s)\leq\varphi_0$ (which follows \eqref{varphiFunk} since the negative exponential is decreasing) in the forcing term. Similarly,
\begin{align*}
\big\vert 
 \big\langle
 \Lambda^\frac{3}{2 } e^{\varphi  \Lambda^{1/2}  }B(\bu,q)_{\mathbb{C}}
,\,      \Lambda^\frac{3}{2 } e^{\varphi  \Lambda^{1/2}  } q
\big\rangle_{\mathbb{C}}
\big\vert
\lesssim&
\Vert \bu\Vert_{H^3_{\mathbb{C}}}
 \Vert q\Vert_{H^3_{\mathbb{C}}}^2 
+
\varphi
\big(
 \Vert e^{\varphi \Lambda^{1/2}} \bu\Vert_{H^{3}_{\mathbb{C}}}
 \Vert e^{\varphi \Lambda^{1/2}}q\Vert_{H^{7/2}_{\mathbb{C}}}^2 
\\&
\quad
+
 \Vert e^{\varphi \Lambda^{1/2}} \bu\Vert_{H^{7/2}_{\mathbb{C}}}
 \Vert e^{\varphi \Lambda^{1/2}}q\Vert_{H^{3}_{\mathbb{C}}}
  \Vert e^{\varphi \Lambda^{1/2}}q\Vert_{H^{7/2}_{\mathbb{C}}}\big)
\\
\lesssim& 
 \Vert q\Vert_{H^3_{\mathbb{C}}}^3 
+
 \Vert f\Vert_{H^2_{\mathbb{C}}}^3 
+
\varphi
\big( 
  \Vert  e^{\varphi \Lambda^{1/2}}q\Vert_{H^{3}_{\mathbb{C}}}
  +
\Vert  e^{\varphi_0 \Lambda^{1/2}}f\Vert_{H^{3}_{\mathbb{C}}}
\big)
\\&\times
\big(
 \Vert ( e^{\varphi \Lambda^{1/2}} b\,,\, e^{\varphi \Lambda^{1/2}}q)\Vert_{  H^{9/2}_{\mathbb{C}} \times H^{7/2}_{\mathbb{C}}}^2
 +
 \Vert  e^{\varphi_0 \Lambda^{1/2}}f \Vert_{    H^{7/2}_{\mathbb{C}}}^2
 \big).
\end{align*}
Finally, we also obtain
\begin{align*}
\big\vert 
 \big\langle
 \Lambda^\frac{3}{2 } e^{\varphi  \Lambda^{1/2}  }&B((\bu-\bu_h),b)_{\mathbb{C}}
,\,   \Lambda^\frac{3}{2 } e^{\varphi  \Lambda^{1/2}  }q
\big\rangle_{\mathbb{C}}
\big\vert
\\\lesssim& 
\big(\Vert \bu\Vert_{H^3_{\mathbb{C}}}
+
\Vert \bu_h\Vert_{H^3_{\mathbb{C}}}
\big)
 \Vert b\Vert_{H^3_{\mathbb{C}}}  
\Vert q\Vert_{H^3_{\mathbb{C}}} 
\\&
+
\varphi
\big( \Vert e^{\varphi \Lambda^{1/2}} \bu\Vert_{H^{7/2}_{\mathbb{C}}}
+
\Vert e^{\varphi_0 \Lambda^{1/2}} \bu_h\Vert_{H^{7/2}_{\mathbb{C}}}
\big)
 \Vert e^{\varphi \Lambda^{1/2}} b\Vert_{H^{3}_{\mathbb{C}}}  
\Vert  e^{\varphi \Lambda^{1/2}} q\Vert_{H^{7/2}_{\mathbb{C}}}
\\&
+
\varphi
\big( \Vert e^{\varphi \Lambda^{1/2}} \bu\Vert_{H^{3}_{\mathbb{C}}}
+
\Vert e^{\varphi_0 \Lambda^{1/2}} \bu_h\Vert_{H^{3}_{\mathbb{C}}}
\big)
 \Vert e^{\varphi \Lambda^{1/2}} b\Vert_{H^{7/2}_{\mathbb{C}}}  
\Vert  e^{\varphi \Lambda^{1/2}} q\Vert_{H^{7/2}_{\mathbb{C}}}
\\
\lesssim&
\Vert b\Vert_{H^3_{\mathbb{C}}}^3
+
 \Vert q\Vert_{H^3_{\mathbb{C}}}^3 
+
 \Vert \bu_h\Vert_{H^3_{\mathbb{C}}}^3 
+
\Vert f\Vert_{H^2_{\mathbb{C}}}^3 
\\&+
\varphi
\big(
 \Vert  e^{\varphi \Lambda^{1/2}} b \Vert_{H^{3}_{\mathbb{C}} } 
  +
 \Vert  e^{\varphi_0  \Lambda^{1/2}  } \bu_h  \Vert_{H_{\mathbb{C}}^{3}  }
\big)
\\&\times 
\big(
 \Vert( e^{\varphi \Lambda^{1/2}} b\,,\,e^{\varphi \Lambda^{1/2}}q)\Vert_{H^{7/2}_{\mathbb{C}} \times H^{7/2}_{\mathbb{C}}}^2 
 +
 \Vert (e^{\varphi_0  \Lambda^{1/2}  } \bu_h \,,\, e^{\varphi_0 \Lambda^{1/2}}f)\Vert_{H_{\mathbb{C}}^{7/2} \times H^{5/2}_{\mathbb{C}}}^2
 \big) .
\end{align*}
Collecting everything yields
\begin{equation}
\begin{aligned}
\label{est0}
\frac{1}{2}\frac{\dd}{\ds} \Vert    (e^{\varphi  \Lambda^{1/2}  } b &\,,\,e^{\varphi  \Lambda^{1/2}  } q ) \Vert^2_{H_{\mathbb{C}}^4\times H_{\mathbb{C}}^3} 
\\
\leq&
\cos\theta\,\Dot{\varphi} 
\Vert    (e^{\varphi  \Lambda^{1/2}  } b  \,,\,e^{\varphi  \Lambda^{1/2}  } q ) \Vert^2_{H_{\mathbb{C}}^{9/2}\times H_{\mathbb{C}}^{7/2}} 
\\&+
c\cos\theta\big(
\Vert( b\,,\, q)\Vert_{H^4_{\mathbb{C}} \times H^3_{\mathbb{C}}}^3 
+
 \Vert (\bu_h\,,\, f)\Vert_{H^3_{\mathbb{C}}\times H^3_{\mathbb{C}}}^3 
\big)
\\&+c\cos\theta\,
\varphi
\big(
 \Vert( e^{\varphi \Lambda^{1/2}} b\,,\,e^{\varphi \Lambda^{1/2}}q)\Vert_{H^{4}_{\mathbb{C}} \times H^{3}_{\mathbb{C}}} 
  +
 \Vert (e^{\varphi_0  \Lambda^{1/2}  } \bu_h \,,\, e^{\varphi_0 \Lambda^{1/2}}f)\Vert_{H_{\mathbb{C}}^{3} \times H^{3}_{\mathbb{C}}}
\big)
\\&\times 
\big(
 \Vert( e^{\varphi \Lambda^{1/2}} b\,,\,e^{\varphi \Lambda^{1/2}}q)\Vert_{H^{9/2}_{\mathbb{C}} \times H^{7/2}_{\mathbb{C}}}^2 
 +
 \Vert (e^{\varphi_0  \Lambda^{1/2}  } \bu_h \,,\, e^{\varphi_0 \Lambda^{1/2}}f)\Vert_{H_{\mathbb{C}}^{7/2} \times H^{7/2}_{\mathbb{C}}}^2
 \big) 
\end{aligned}
\end{equation}
where, due to \eqref{forcingReg} and \cite{crisan2023theoretical1} we have that for all  $\mathrm{Re}(\zeta)\in I$,
\begin{align*}
\Vert( b(\zeta)&\,,\, q(\zeta))\Vert_{H^4_{\mathbb{C}} \times H^3_{\mathbb{C}}}^3 
+
 \Vert (\bu_h\,,\, f)\Vert_{H^3_{\mathbb{C}}\times H^3_{\mathbb{C}}}^3 
\\&
\lesssim \Vert( b_0\,,\, q_0)\Vert_{H^4_{\mathbb{C}} \times H^3_{\mathbb{C}}}^3 
+
 \Vert (\bu_h\,,\, f)\Vert_{H^4_{\mathbb{C}}\times H^3_{\mathbb{C}}}^3.
\end{align*} 
However, since we can use $\vert\bk\vert=\tfrac{1}{2\varphi_0}2\varphi_0\vert\bk\vert\leq \tfrac{1}{2\varphi_0}e^{2\varphi_0\vert\bk\vert}$ to obtain
\begin{align*}
\Vert \bu_h\Vert_{H^4}^3
\leq
\frac{1}{(2\varphi_0)^{3/2}}
\Big(\sum_{\bk\in \mathbb{Z}^d_{\bm{0}}}  \vert \bk\vert^{7}e^{2\varphi_0\vert\bk\vert} \vert \hat{\bu}_{h,\bk} \vert^2\Big)^{3/2}
=
\frac{1}{(2\varphi_0)^{3/2}}
\Vert e^{\varphi_0\Lambda^{1/2}} \bu_h\Vert_{H^{7/2}}^3,
\end{align*}
it follows that for
\begin{align*}
\mathcal{D}(\mathrm{data}):=\Vert    (e^{\varphi_0  \Lambda^{1/2}  } b_0 \,,\,e^{\varphi_0   \Lambda^{1/2}  } q_0  ) \Vert_{H_{\mathbb{C}}^4\times H_{\mathbb{C}}^3}^2
+
\Vert    (e^{\varphi_0  \Lambda^{1/2}  } \bu_h \,,\,e^{\varphi_0   \Lambda^{1/2}  } f  ) \Vert_{H_{\mathbb{C}}^{7/2} \times H^{7/2}_{\mathbb{C}}}^2,
\end{align*} 
we have that
\begin{align*}
\Vert( b(\zeta)&\,,\, q(\zeta))\Vert_{H^4_{\mathbb{C}} \times H^3_{\mathbb{C}}}^3 
+
 \Vert (\bu_h\,,\, f)\Vert_{H^3_{\mathbb{C}}\times H^3_{\mathbb{C}}}^3  
\lesssim  c\,\mathcal{D}(\mathrm{data})^\frac{3}{2}
\end{align*}
holds for all  $\mathrm{Re}(\zeta)\in I$. Thus, it follows from \eqref{est0} that
\begin{align*}
\frac{1}{2}\frac{\dd}{\ds} \Vert    (e^{\varphi  \Lambda^{1/2}  } b \,,\,e^{\varphi  \Lambda^{1/2}  } q ) \Vert^2_{H_{\mathbb{C}}^4\times H_{\mathbb{C}}^3} 
&\leq
c\mathcal{D}(\mathrm{data})^\frac{3}{2}\cos\theta 
+\cos\theta(\Dot{\varphi} +c\,
\varphi
\Theta(s))\Gamma(s) 
\end{align*}
where
\begin{align*}
\Gamma(\zeta):= \Vert( e^{\varphi \Lambda^{1/2}} b\,,\,e^{\varphi \Lambda^{1/2}}q)\Vert_{H^{9/2}_{\mathbb{C}} \times H^{7/2}_{\mathbb{C}}}^2 
 +
 \Vert (e^{\varphi_0  \Lambda^{1/2}  } \bu_h \,,\, e^{\varphi_0 \Lambda^{1/2}}f)\Vert_{H_{\mathbb{C}}^{7/2} \times H^{7/2}_{\mathbb{C}}}^2.
\end{align*} 
Since \eqref{varphiFunk} holds, it follows that
\begin{align*}
\Dot{\varphi}(s)=-c\,
\Theta(s)
\varphi_0e^{-c\int_0^s \Theta(\tau)\dd\tau}=-c\,\Theta(s)
\varphi(s).
\end{align*}
Thus, it follow from integration that 
\begin{align*} 
\Vert  &  (e^{\varphi(s\cos\theta)  \Lambda^{1/2}  } b(se^{i\theta}) \,,\,e^{\varphi(s\cos\theta)   \Lambda^{1/2}  } q(se^{i\theta})  ) \Vert^2_{H_{\mathbb{C}}^4\times H_{\mathbb{C}}^3} 
\\&\leq
\Vert    (e^{\varphi_0  \Lambda^{1/2}  } b_0 \,,\,e^{\varphi_0   \Lambda^{1/2}  } q_0  ) \Vert^2_{H_{\mathbb{C}}^4\times H_{\mathbb{C}}^3} 
+
c\,s \mathcal{D}(\mathrm{data})^\frac{3}{2}\cos\theta 
\end{align*}
holds for any $\mathrm{Re}(\zeta)\in I$. As such, 
if
\begin{align*}
s \mathcal{D}(\mathrm{data})^\frac{1}{2}\leq \frac{1}{c\cos\theta}
\end{align*} 
we obtain 
\begin{align*} 
\Vert  &  (e^{\varphi(s\cos\theta)  \Lambda^{1/2}  } b(se^{i\theta}) \,,\,e^{\varphi(s\cos\theta)   \Lambda^{1/2}  } q(se^{i\theta})  ) \Vert^2_{H_{\mathbb{C}}^4\times H_{\mathbb{C}}^3} 
 \leq
\mathcal{D}(\mathrm{data})
\end{align*}
so that the domain of analyticity is the set
\begin{align*}
\left\{ \zeta=s e^{i\theta}\,:\,\quad  s\cos\theta\in I, \quad
\vert\theta\vert<\frac{\pi}{2}, \quad  0<s \,\mathcal{D}(\mathrm{data})^\frac{1}{2}<  \frac{1}{c\cos\theta}\right\}.
\end{align*} 

\section{Appendix} 
We devote this section to the proof of Lemma \ref{lem:ConvEst} which follows the arguments of  \cite[Lemma 8]{levermore1997analyticity}.  We rely on two preliminary results, the first of which is given by \cite[Lemma 9]{levermore1997analyticity}.
\begin{lemma}
\label{lem:algebraIneq}
Let $r\geq1$ and $\varphi\geq0$. Then for all all real numbers $\xi,\eta\geq0$, the estimate
\begin{align*}
\big\vert \xi^r e^{\varphi\xi}-  \eta^r e^{\varphi\eta} \big\vert
\lesssim_r \vert \xi-\eta\vert \big[(\vert \xi-\eta\vert^{r-1}+\eta^{r-1})
+
\varphi(\vert \xi-\eta\vert^r+\eta^r) e^{\varphi(\vert \xi-\eta\vert+\eta)} \big]
\end{align*}
holds.
\end{lemma}

The second preliminary result is the following.
\begin{lemma}
\label{lem:sum}
For $j:=\vert\bj\vert$, the sequence $j^{-2(r-\frac{3}{2})}$ is summable over $\bj\in \mathbb{Z}^d_{\bm{0}}= 2 \pi \mathbb{Z}^d\setminus\{\bm{0}\}$ for $r>\frac{d}{2}+\frac{3}{2}$.
\end{lemma}
\begin{proof}
In the following, we work in dimension $d=2$ for simplicity so that $r>\frac{5}{2}$. The extension to arbitrary dimensions follows by modifying $N(s)$ below.
\\
Let $\bj  \in \mathbb{Z}^2_{\bm{0}}= 2 \pi \mathbb{Z}^2\setminus\{\bm{0}\}$ with $\vert \bj\vert =j$. If we consider $N(s)=\pi s^2 +\mathcal{O}(s)$, the number of lattice points in the circle of radius $s>0$ (Gauss circle problem), then since $j\geq1$, we have that
\begin{align*}
\sum_{\bj \in \mathbb{Z}^2_{\bm{0}}} j^{-2(r-\frac{3}{2})} 
=
\lim_{R\rightarrow0}\sum_{\substack{ \bj \in \mathbb{Z}^2_{\bm{0}} \\  1\leq j^2\leq R}} j^{-2(r-\frac{3}{2})} 
=
\lim_{R\rightarrow0}\int_1^{\sqrt{R}} s^{-2(r-\frac{3}{2})}\dd N(s).
\end{align*}
However, for $r>\frac{5}{2}$,
\begin{align*}
\int_1^{\sqrt{R}} s^{-2(r-\frac{3}{2})}\dd N(s)
&=
s^{-2(r-\frac{3}{2})} N(s)\bigg\vert_1^{\sqrt{R}} 
+
2\big(r-\frac{3}{2}\big)
\int_1^{\sqrt{R}} s^{-2(r-1)}N(s)\ds
\\
&=
 \pi R^{-(r-\frac{5}{2})}
 +
2\big(r-\frac{3}{2}\big)
 \left[ \frac{-\pi}{2r-5} s^{-2(r-\frac{5}{2})}   \right]_1^{\sqrt{R}}
  + \mathcal{O}(1)
 \\
&=
 \pi R^{-(r-\frac{5}{2})} 
+
\frac{-\pi(2r-3)}{2r-5} \left[  R^{-(r-\frac{5}{2})} -1 \right]  + \mathcal{O}(1).
\end{align*}
Because $r>\frac{5}{2}$, passing to the limit $R\rightarrow\infty$, yields
\begin{align*}
\sum_{\bj \in \mathbb{Z}^2_{\bm{0}}} j^{-2(r-\frac{3}{2})} = \frac{\pi(2r-3)}{2r-5}  + \mathcal{O}(1).
\end{align*}
and thus, summable.
\end{proof}
Getting back to the proof of Lemma \ref{lem:ConvEst}, we begin by setting $\bz:= \Lambda^{r/2} e^{\varphi  A^{1/2}  }\bw$ so that
\begin{align*}
\big\langle
\Lambda^{r/2}   e^{\varphi  \Lambda^{1/2}  }
     ( \bu  \cdot\nabla) \bv
,\,   \Lambda^{r/2}   e^{\varphi  \Lambda^{1/2}  }\bw
\big\rangle 
=&
I_1+I_2
\end{align*}
where
\begin{align*}
I_1=&
\big\langle
    ( \bu  \cdot\nabla)(
 \Lambda^{r/2}   e^{\varphi  \Lambda^{1/2}  } \bv)
,\,   \Lambda^{r/2} e^{\varphi  A^{1/2}  }\bw
\big\rangle 
\\
I_2=&
\big\langle 
     ( \bu  \cdot\nabla) \bv
,\,   \Lambda^{r/2}   e^{\varphi  \Lambda^{1/2}  }\bz 
\big\rangle 
-I_1.
\end{align*} 
To simplify notations, for $\bj,\bk,\bl \in \mathbb{Z}^d_{\bm{0}}= 2 \pi \mathbb{Z}^d\setminus\{\bm{0}\}$,
we let  $j:=\vert\mathbf{j}\vert$,  $k:=\vert\mathbf{k}\vert$ and
 $\ell:=\vert\bl\vert$ and note that for $\bn \in 2\pi \mathbb{Z}^d$,
\begin{align*}
\frac{1}{(2\pi)^d}\int_{\mathbb{T}^d}e^{i\bn\cdot\bx}\dx
= 
  \begin{cases}
    0      & \quad \text{if } \bn\neq\bm{0}\\
    1  & \quad \text{if } \bn=\bm{0}.
  \end{cases}
\end{align*}
With this information, we are able to rewrite $I_1$ as follows:
\begin{align*}
I_1&=
\Big\langle 
     \sum_{\bj}  \hat{\bu}_{\bj} e^{i\bj\cdot\bx} \cdot \sum_{\bk}i\bk\,k^re^{\varphi  k }\hat{\bv}_\bk e^{i\bk\cdot\bx}
\,,\,  
 \sum_{\bl}\ell^re^{\varphi  \ell  }\hat{\bw}_{\bl} e^{i\bl\cdot\bx} 
\Big\rangle 
\\&=i(2\pi)^d
\sum_{\bn}
\sum_{\bj+\bk+\bl=\bn}
  \hat{\bu}_{\bj} \cdot  \bk\, k^re^{\varphi  k }\ell^re^{\varphi  \ell  }\hat{\bv}_\bk\cdot\hat{\bw}_{\bl}\frac{1}{(2\pi)^d}\int_{\mathbb{T}^d} e^{i\bn\cdot\bx} \dx 
  \\&=i (2\pi)^d
\sum_{\bj+\bk+\bl=\bm{0}}
   \hat{\bu}_{\bj} \cdot  \bk\, k^re^{\varphi  k }\ell^re^{\varphi  \ell  }\hat{\bv}_\bk\cdot\hat{\bw}_{\bl}.
\end{align*}
%
However, note that $\vert\bk\vert\leq k^\frac{1}{2}(\ell j)^\frac{1}{2} $ (since $j,k,\ell\geq1$, it follows that $k \leq k\ell$ ) and so
\begin{align*}
\vert\bk\vert k^r\ell^r\leq j^\frac{1}{2}k^{r+\frac{1}{2}}\ell^{r+\frac{1}{2}}
=j^{r-1}j^{-(r-\frac{3}{2})}k^{r+\frac{1}{2}}\vert\bj+\bk\vert^{r+\frac{1}{2}}.
\end{align*}
Consequently, it follows that
\begin{align*}
\vert I_1\vert &\lesssim
\sum_{\bj} j^{r-1}j^{-(r-\frac{3}{2})}
\hat{\bu}_{\bj}
\sum_{\bj+\bk \neq\bm{0}}
 k^{r+\frac{1}{2}}e^{\varphi  k }\vert\hat{\bv}_\bk\vert
\vert\bj+\bk\vert^{r+\frac{1}{2}}e^{\varphi  \vert\bj+\bk\vert}\vert\hat{\bw}_{ \bj+\bk }\vert
   \\&
   \lesssim
    \Big(  \sum_{\bj} j^{-2(r-\frac{3}{2})}  \Big)^\frac{1}{2}
  \Big(  \sum_{\bj}j^{2(r-1)}  \vert  \hat{\bu}_{\bj}\vert^2 \Big)^\frac{1}{2}
\Big(\sum_{ \bk}
    k^{2(r+\frac{1}{2})} e^{2\varphi  k }\vert\hat{\bv}_\bk\vert^2\Big)^\frac{1}{2}
 \\&
 \qquad\quad\times
 \Big(\sum_{\bj+\bk \neq\bm{0}}  \vert\bj+\bk\vert^{2(r+\frac{1}{2})}e^{2\varphi \vert\bj+\bk\vert}\vert\hat{\bw}_{ \bj+\bk }\vert^2\Big)^\frac{1}{2} 
  \\&\lesssim
    \varphi \Vert   e^{\varphi  \Lambda^{1/2}  }\bu\Vert _{H^{r}}
\Vert   e^{\varphi  \Lambda^{1/2}  }\bv \Vert _{H^{r+\frac{1}{2}}}   \Vert     e^{\varphi  \Lambda^{1/2}  }\bw \Vert _{H^{r+\frac{1}{2}} }
\end{align*}
where in the last step, we use the summability of $ j^{-2(r-\frac{3}{2})}$ (see Lemma \ref{lem:sum}) and also use the estimate $1\lesssim xe^x$ for $x=\varphi j>0$ inside the $\hat{\bu}_{\bj}$-term . 
We now estimate $I_2$. 
Similar to $I_1$, we obtain
\begin{align*}
\big\langle 
     ( \bu  \cdot\nabla) \bv
&\,,\,   \Lambda^{r/2}   e^{\varphi  \Lambda^{1/2}  }\bz 
\big\rangle 
\\&=
\Big\langle 
     \sum_{\bj}  \hat{\bu}_{\bj} e^{i\bj\cdot\bx} \cdot \sum_{\bk}i\bk\,\hat{\bv}_\bk e^{i\bk\cdot\bx}
\,,\,  
 \sum_{\bl}\ell^re^{\varphi  \ell  }\hat{\bz}_{\bl} e^{i\bl\cdot\bx} 
\Big\rangle 
\\&=i(2\pi)^d
\sum_{\bn}
\sum_{\bj+\bk+\bl=\bn}
  \hat{\bu}_{\bj} \cdot  \bk\,\hat{\bv}_\bk\cdot \ell^re^{\varphi  \ell  }\hat{\bz}_{\bl}\frac{1}{(2\pi)^d}\int_{\mathbb{T}^d} e^{i\bn\cdot\bx} \dx 
  \\&=i (2\pi)^d
\sum_{\bj+\bk+\bl=\bm{0}}
   \hat{\bu}_{\bj}  \cdot  \bk\,\hat{\bv}_\bk\cdot \ell^re^{\varphi  \ell  }\hat{\bz}_{\bl}.
\end{align*}
Subtracting the identity for $I_1$ from this results in
\begin{align*}
I_2 =
i (2\pi)^d
\sum_{\bj+\bk+\bl=\bm{0}}
  \hat{\bu}_{\bj}   \cdot  \bk\,\hat{\bv}_\bk \cdot\hat{\bz}_{\bl}(\ell^re^{\varphi  \ell  }-k^re^{\varphi  k  }) 
\end{align*} 
where since $\bz= \Lambda^{r/2} e^{\varphi  \Lambda^{1/2}  }\bw$, we have
\begin{align*}
\vert \hat{\bz}_{\bl}\vert= \ell^{r} e^{\varphi  \ell}\vert \hat{\bw}_{\bl}\vert \lesssim  \ell^{r}\vert \hat{\bw}_{\bl}\vert +  \ell^{r}e^{\varphi  \ell}\vert \hat{\bw}_{\bl}\vert\,\varphi (j+k)
=
 \ell^{r}\vert \hat{\bw}_{\bl}\vert +   \varphi (j+k)\vert \hat{\bz}_{\bl}\vert.
\end{align*} 
We also note that by using Lemma \ref{lem:algebraIneq}, we obtain
\begin{align*}
 \big\vert
\ell^re^{\ell} 
-
k^re^{k} \big\vert
&\lesssim
\vert \ell-k\vert \big[(\vert \ell-k\vert^{r-1} +   k^{r-1}) +\varphi(\vert \ell-k\vert^r + k^r)e^{\varphi(\vert \ell-k\vert +k)}  \big]
\\
&\lesssim
j \big[(j^{r-1}+k^{r-1}) +\varphi(j^r +k^r)e^{\varphi(j+k)}  \big]
\end{align*}
Combining the two relations yields
\begin{align*}
\vert \hat{\bz}_{\bl}\vert
\,  \big\vert
\ell^re^{\ell} 
-
k^re^{k} \big\vert
&\lesssim
j\,\vert \hat{\bz}_{\bl}\vert
 \big[(j^{r-1}+k^{r-1}) +\varphi(j^r +k^r)e^{\varphi(j+k)}  \big]
\\&\lesssim
j\big( \ell^{r}\vert \hat{\bw}_{\bl}\vert +   \varphi (j+k)\vert \hat{\bz}_{\bl}\vert\big)
(j^{r-1}+k^{r-1})
 +
 \varphi\,j\,\vert \hat{\bz}_{\bl}\vert(j^r +k^r)e^{\varphi(j+k)}  
 \\&\lesssim_r
j \ell^{r}\vert \hat{\bw}_{-(\mathbf{j}+\mathbf{k})}\vert  (j^{r-1}+k^{r-1}) 
 +\varphi\,j\,\vert \hat{\bz}_{-(\mathbf{j}+\mathbf{k})}\vert(j^r +k^r)e^{\varphi(j+k)}  
\end{align*}
where in the last line, we have used Young's inequality to obtain
\begin{align*}
(j+k)(j^{r-1}+k^{r-1})\lesssim_r (j^r+k^r) \lesssim_r(j^r+k^r)e^{\varphi(j+k)}.
\end{align*}
as well as the identity $\hat{\bz}_{\bl}=\hat{\bz}_{-(\mathbf{j}+\mathbf{k})}$ which holds since $\mathbf{j}+\mathbf{k}+\bl=\bm{0}$.
Consequently, we obtain
\begin{align*}
I_2\lesssim&
\sum_{ 
\mathbf{j}+\mathbf{k}+\bl=\bm{0} } 
k \vert\hat{\bu}_\mathbf{j}\vert  \vert \hat{\bv}_\mathbf{k}\vert
\,\vert \hat{\bz}_{\bl}\vert
\,  \big\vert
\ell^re^{\ell} 
-
k^re^{k} \big\vert\lesssim I_2^A+I_2^B
\end{align*}
where
\begin{align*}
 I_2^A&:=
\sum_{ 
\mathbf{j}+\mathbf{k}+\bl=\bm{0} } 
\ell^r \vert\hat{\bu}_\mathbf{j}\vert  \vert \hat{\bv}_\mathbf{k}\vert\,\vert \hat{\bw}_{-(\mathbf{j}+\mathbf{k})}\vert
(j^rk+k^rj),
\\
 I_2^B&:=
\sum_{ 
\mathbf{j}+\mathbf{k}+\mathbf{l}=\bm{0}
 } 
\varphi  \vert\hat{\bu}_\mathbf{j}\vert  \vert \hat{\bv}_\mathbf{k}\vert\,\vert \hat{\bz}_{-(\mathbf{j}+\mathbf{k})}\vert
(j^{r+1}k+k^{r+1}j)  e^{\varphi(j+k)}.
\end{align*}
Now given that $\vert\hat{\bw}_{ -(\mathbf{j}+\mathbf{k})}\vert = \vert\overline{\hat{\bw}}_{ \mathbf{j}+\mathbf{k}}\vert=\vert \hat{\bw}_{ \mathbf{j}+\mathbf{k}}\vert$, we note that
\begin{align*}
I_2^A&\leq
\sum_{\bk} 
k \,\vert \hat{\bv}_\mathbf{k}\vert
\sum_{\bj+\bk \neq\bm{0}}
\vert \mathbf{j}+\mathbf{k}\vert^r \vert\hat{\bu}_\mathbf{j}\vert  \vert \hat{\bw}_{ \mathbf{j}+\mathbf{k}}\vert
j^r
\\&+
\sum_{\bj} 
j  \vert\hat{\bu}_\mathbf{j}\vert  
\sum_{\bj+\bk \neq\bm{0}}
\vert \mathbf{j}+\mathbf{k}\vert^r
\,\vert \hat{\bv}_\mathbf{k}\vert
\vert \hat{\bw}_{ \mathbf{j}+\mathbf{k}}\vert
k^r
\end{align*}
where
\begin{align*}
\sum_{\bk} &
k \,\vert \hat{\bv}_\mathbf{k}\vert
\sum_{\bj+\bk \neq\bm{0}} 
\vert \mathbf{j}+\mathbf{k}\vert^r \vert\hat{\bu}_\mathbf{j}\vert  \vert \hat{\bw}_{ \mathbf{j}+\mathbf{k}}\vert
j^r
\\&\lesssim
\Big(
\sum_{\bk} 
k^{2(1-r)}
\Big)^\frac{1}{2}
\Big(
\sum_{\bk} 
k^{2r} \,\vert \hat{\bv}_\mathbf{k}\vert^2
\Big)^\frac{1}{2}
\Big(
\sum_{\bj} 
j^{2r} \vert\hat{\bu}_\mathbf{j}\vert^2  \Big)^\frac{1}{2}
\Big(
\sum_{\bj+\bk \neq\bm{0}}
\vert \mathbf{j}+\mathbf{k}\vert^{2r}
\vert \hat{\bw}_{ \mathbf{j}+\mathbf{k}}\vert^2
\Big)^\frac{1}{2}
\\&\lesssim
\Big(
\sum_{\bk} 
k^{2r} \,\vert \hat{\bv}_\mathbf{k}\vert^2
\Big)^\frac{1}{2}
 \Big(
\sum_{\bj} 
j^{2r}\vert\hat{\bu}_\mathbf{j}\vert^2 
\Big)^\frac{1}{2} 
\Big(
\sum_{\bj+\bk \neq\bm{0}}
\vert \mathbf{j}+\mathbf{k}\vert^{2r}
\vert \hat{\bw}_{ \mathbf{j}+\mathbf{k}}\vert^2
\Big)^\frac{1}{2}
\\&\lesssim
\Vert \bv\Vert_{H^r}
 \Vert\bu\Vert_{H^r}  
\Vert \bw\Vert_{H^r}.
\end{align*} 
The same estimate holds for the second summand in $I_2^A$ so that we have
\begin{align*}
I_2^A 
&\lesssim
\Vert \bv\Vert_{H^r}
 \Vert\bu\Vert_{H^r}  
\Vert \bw\Vert_{H^r}.
\end{align*} 
Next, we rewrite $I_2^B$ as
\begin{align*}
 I_2^B:=&
\sum_{ 
\mathbf{j}+\mathbf{k}+\bl=\bm{0} } 
\varphi  \vert\hat{\bu}_\mathbf{j}\vert   \vert \hat{\bv}_\mathbf{k}\vert\,\vert \hat{\bz}_{\mathbf{j}+\mathbf{k}}\vert
j^{r+1}k e^{\varphi(j+k)}
\\
&+
\sum_{ 
\mathbf{j}+\mathbf{k}+\bl=\bm{0} } 
\varphi  \vert\hat{\bu}_\mathbf{j}\vert  \vert \hat{\bv}_\mathbf{k}\vert\,\vert \hat{\bz}_{\mathbf{j}+\mathbf{k}}\vert
k^{r+1}j  e^{\varphi(j+k)}.
\end{align*}
 Now, by recalling that $\bz= \Lambda^{r/2} e^{\varphi  \Lambda^{1/2}  }\bw$ and $j^\frac{1}{2}\leq \ell^\frac{1}{2}k^\frac{1}{2}$ (note that $j,k,\ell\geq1$) , we obtain
\begin{align*}
\sum_{ 
\mathbf{j}+\mathbf{k}+\bl=\bm{0} }&
\varphi \vert\hat{\bu}_\mathbf{j}\vert   \vert \hat{\bv}_\mathbf{k}\vert\,\vert \hat{\bz}_{\mathbf{j}+\mathbf{k}}\vert
j^{r+1}k e^{\varphi(j+k)}
\\&\leq
\varphi
\sum_{\bk} 
 k^\frac{3}{2} e^{\varphi k}  \vert \hat{\bv}_\mathbf{k}\vert
\sum_{\bj+\bk \neq\bm{0}}
 j^{r +\frac{1}{2}} e^{\varphi j}
 \vert\hat{\bu}_\mathbf{j}\vert \vert\mathbf{j}+\mathbf{k}\vert^{r+\frac{1}{2}}e^{\varphi \vert\mathbf{j}+\mathbf{k}\vert} \vert \hat{\bw}_{\mathbf{j}+\mathbf{k}}\vert
\\&
\lesssim
\varphi
\Big(
\sum_{\bk} 
\frac{1}{k^{2(r-\frac{3}{2})}}
\Big)^\frac{1}{2}
\Big(
\sum_{\bk} 
k^{2r} e^{2\varphi k}\vert \hat{\bv}_\mathbf{k}\vert^2
\Big)^\frac{1}{2}
\Big(
\sum_{\bj} 
j^{2(r+\frac{1}{2})}e^{2\varphi j} \vert\hat{\bu}_\mathbf{j}\vert^2 \Big)^\frac{1}{2}
\\&
\qquad\quad\times
\Big(
\sum_{\bj+\bk \neq\bm{0}}
\vert \mathbf{j}+\mathbf{k}\vert^{2(r+\frac{1}{2})}e^{2\varphi \vert \mathbf{j}+\mathbf{k}\vert}
\vert \hat{\bw}_{ \mathbf{j}+\mathbf{k}}\vert^2
\Big)^\frac{1}{2}
\\&\lesssim
\varphi
\Vert e^{\varphi \Lambda^{1/2}} \bv\Vert_{H^r}
 \Vert e^{\varphi \Lambda^{1/2}}\bu\Vert_{H^{r+\frac{1}{2}}}  
\Vert  e^{\varphi \Lambda^{1/2}}\bw\Vert_{H^{r+\frac{1}{2}}}.
\end{align*}
 for $r>2$. Similarly, we can use  $k^\frac{1}{2}\leq \ell^\frac{1}{2}j^\frac{1}{2}$ to obtain
\begin{align*}
\sum_{ 
\mathbf{j}+\mathbf{k}+\bl=\bm{0} } &
\varphi  \vert\hat{\bu}_\mathbf{j}\vert   \vert \hat{\bv}_\mathbf{k}\vert\,\vert \hat{\bz}_{\mathbf{j}+\mathbf{k}}\vert
k^{r+1}j  e^{\varphi(j+k)}
\\&\leq
\varphi
\sum_{\bj} 
 j^{\frac{3}{2}} e^{\varphi j}
 \vert\hat{\bu}_\mathbf{j}\vert
 \sum_{\bj+\bk \neq\bm{0}}
 k^{r+\frac{1}{2}} e^{\varphi k}  \vert \hat{\bv}_\mathbf{k}\vert
 \vert\mathbf{j}+\mathbf{k}\vert^{r +\frac{1}{2}}e^{\varphi \vert\mathbf{j}+\mathbf{k}\vert} \vert \hat{\bw}_{\mathbf{j}+\mathbf{k}}\vert
\\&
\lesssim
\varphi
\Big(
\sum_{\bj} 
\frac{1}{j^{2(r-\frac{3}{2})}}
\Big)^\frac{1}{2}
\Big(
\sum_{\bj} 
j^{2r} e^{2\varphi j}\vert \hat{\bu}_\mathbf{j}\vert^2
\Big)^\frac{1}{2}
\Big(
\sum_{\bk} 
k^{2(r+\frac{1}{2})}e^{2\varphi k} \vert\hat{\bv}_\mathbf{k}\vert^2 \Big)^\frac{1}{2}
\\&
\qquad\quad\times
\Big(
\sum_{\bj+\bk \neq\bm{0}}
\vert \mathbf{j}+\mathbf{k}\vert^{2(r+\frac{1}{2})}e^{2\varphi \vert \mathbf{j}+\mathbf{k}\vert}
\vert \hat{\bw}_{ \mathbf{j}+\mathbf{k}}\vert^2
\Big)^\frac{1}{2}
\\&\lesssim
\varphi
\Vert e^{\varphi \Lambda^{1/2}} \bu\Vert_{H^{r }}
 \Vert e^{\varphi \Lambda^{1/2}}\bv\Vert_{H^{r+\frac{1}{2}}}  
\Vert  e^{\varphi \Lambda^{1/2}}\bw\Vert_{H^{r+\frac{1}{2}}}
\end{align*}
so that
\begin{align*}
I_2^B
\lesssim&
\varphi
\Big(
\Vert e^{\varphi \Lambda^{1/2}} \bv\Vert_{H^r}
 \Vert e^{\varphi \Lambda^{1/2}}\bu\Vert_{H^{r+\frac{1}{2}}}  
+
\Vert e^{\varphi \Lambda^{1/2}} \bu\Vert_{H^{r }}
 \Vert e^{\varphi \Lambda^{1/2}}\bv\Vert_{H^{r+\frac{1}{2}}} 
 \Big) 
\Vert  e^{\varphi \Lambda^{1/2}}\bw\Vert_{H^{r+\frac{1}{2}}}.
\end{align*} 
Combining the estimates for $\vert I_1\vert$, $I_2^A$ and $I_2^B$ above yields the desired result.
\section*{Statements and Declarations} 
\subsection*{Author Contribution}
The author wrote and reviewed the manuscript.
\subsection*{Conflict of Interest}
The author declares that they have no conflict of interest.
\subsection*{Data Availability Statement}
Data sharing is not applicable to this article as no datasets were generated
or analyzed during the current study.
\subsection*{Competing Interests}
The author have no competing interests to declare that are relevant to the content of this article.
%
%
%

\end{document}